\documentclass{amsart}
\usepackage{amssymb}
\copyrightinfo{2013}{V.X. Genest \emph{et al.}}

\newtheorem{theorem}{Theorem}

\newtheorem{lemma}{Lemma}

\theoremstyle{definition}

\theoremstyle{remark}
\newtheorem{remark}[theorem]{Remark}

\numberwithin{equation}{section}

\newcommand{\ket}[1]{\rvert#1\rangle}
\newcommand{\bra}[1]{\langle #1\rvert}
\newcommand{\braket}[2]{\langle #1\rvert#2\rangle}
\newcommand{\Braket}[3]{\bra{#1}\,#2\,\ket{#3}}
\begin{document}
\title[$\mathfrak{osp}_{q}(1|2)$ and $q$-Bannai--Ito polynomials]{The quantum superalgebra $\mathfrak{osp}_{q}(1|2)$ and a $q$-generalization of the Bannai--Ito polynomials}
\author[V.X. Genest]{Vincent X. genest}
\address{Centre de Recherches Math\'ematiques, Universit\'e de Montr\'eal, P.O. Box 6128, Centre-ville Station, Montr\'eal (Qu\'ebec) H3C 3J7}
\email{genestvi@crm.umontreal.ca}
\author[L. Vinet]{Luc Vinet}
\address{Centre de Recherches Math\'ematiques, Universit\'e de Montr\'eal, P.O. Box 6128, Centre-ville Station, Montr\'eal (Qu\'ebec) H3C 3J7}
\email{vinetl@crm.umontreal.ca}
\author[A. Zhedanov]{Alexei Zhedanov}
\address{Donetsk Institute for Physics and Technology, Donetsk 83114, Ukraine}
\email{zhedanov@yahoo.com}
\date{}
\dedicatory{}
\begin{abstract}
The Racah problem for the quantum superalgebra $\mathfrak{osp}_{q}(1|2)$ is considered. The intermediate Casimir operators are shown to realize a $q$-deformation of the Bannai--Ito algebra. The Racah coefficients of $\mathfrak{osp}_q(1|2)$ are calculated explicitly in terms of basic orthogonal polynomials that $q$-generalize the Bannai--Ito polynomials.  The relation between these $q$-deformed Bannai--Ito polynomials and the $q$-Racah/Askey-Wilson polynomials is discussed.
\end{abstract}
\maketitle
\section{Introduction}
The goal of this paper is to examine the Racah problem for the quantum superalgebra $\mathfrak{osp}_{q}(1|2)$ and to present a $q$-extension of the Bannai--Ito polynomials.

The Bannai--Ito (BI) polynomials were first introduced by Bannai and Ito in their complete classification of the orthogonal polynomials possessing the Leonard duality property \cite{1984_Bannai&Ito}. The BI polynomials, denoted by $B_{n}(x)$, depend on four real parameters $\rho_1$, $\rho_2$, $r_1$, $r_2$ and can be defined by the three-term recurrence relation
\begin{align}
\label{BI-Recurrence}
 xB_{n}(x)=B_{n+1}(x)+(\rho_1-A_n-C_n)\,B_{n}(x)+A_{n-1}C_{n}\,B_{n-1}(x),
\end{align}
with $B_{-1}(x)=0$, $B_{0}(x)=1$ and where the recurrence coefficients read
\begin{align}
\label{BI-Recurrence-Coefficients}
\begin{aligned}
 A_{n}&=
 \begin{cases}
  \frac{(n+2\rho_1-2r_1+1)(n+2\rho_1-2r_2+1)}{4(n+\kappa+1/2)} & \text{$n$ even}
  \\
  \frac{(n+2\kappa)(n+2\rho_1+2\rho_2+1)}{4(n+\kappa+1/2)} & \text{$n$ odd}
 \end{cases},
 \\
 C_{n}&=
 \begin{cases}
  -\frac{n(n-2r_1-2r_2)}{4(n+\kappa-1/2)} & \text{$n$ even}
  \\
  -\frac{(n+2\rho_2-2r_1)(n+2\rho_2-2r_2)}{4(n+\kappa-1/2)} & \text{$n$ odd}
 \end{cases},
 \end{aligned}
\end{align}
with $\kappa=\rho_1+\rho_2-r_1-r_2+1/2$. The polynomials $B_{n}(x)$ can be obtained as $q\rightarrow -1$ limits of the $q$-Racah \cite{1984_Bannai&Ito} or of the Askey-Wilson \cite{2012_Tsujimoto&Vinet&Zhedanov_AdvMath_229_2123}  polynomials, which sit at the top of the Askey scheme of hypergeometric orthogonal polynomials \cite{2010_Koekoek_&Lesky&Swarttouw}. The Bannai--Ito polynomials are eigenfunctions of the most general self-adjoint first-order shift operator with reflections preserving the space of polynomials of a given degree  \cite{2012_Tsujimoto&Vinet&Zhedanov_AdvMath_229_2123}. Up to affine transformations, this operator has the expression
\begin{align}
\label{BI-Operator}
 \mathcal{L}=D(x) (1-R)+E(x) (T^{+}R-1)+\kappa,
\end{align}
with $D(x)$ and $E(x)$ given by
\begin{align*}
 D(x)=\frac{(x-\rho_1)(x-\rho_2)}{x},\quad E(x)=\frac{(x-r_1+1/2)(x-r_2+1/2)}{x+1/2},
\end{align*}
and where $T^{+}f(x)=f(x+1)$ is the shift operator and $Rf(x)=f(-x)$ is the reflection operator. The BI polynomials satisfy the eigenvalue equation \cite{2012_Tsujimoto&Vinet&Zhedanov_AdvMath_229_2123}
\begin{align}
\label{BI-Eigenvalue}
\mathcal{L}\,B_{n}(x)=(-1)^{n}(n+\kappa)\,B_{n}(x),\quad n=0,1,2,\ldots
\end{align}

There is an algebraic structure associated to the BI polynomials which is called the Bannai--Ito algebra \cite{2012_Tsujimoto&Vinet&Zhedanov_AdvMath_229_2123}. It is defined as the associative algebra over $\mathbb{C}$ with generators $A_1, A_2, A_3$ obeying the relations 
\begin{align}
\label{BI-Algebra}
 \{A_3,A_1\}=A_2+\omega_2,\quad \{A_1,A_2\}=A_3+\omega_3,\quad \{A_2,A_3\}=A_1+\omega_1,
\end{align}
where $\{x,y\}=xy+yx$ is the anticommutator and where $\omega_1,\omega_2,\omega_3$ are complex structure constants. It is clear that in \eqref{BI-Algebra} only two of the generators are genuinely independent. The relation between the algebra \eqref{BI-Algebra} and the polynomials $B_{n}(x)$ is established by noting that the operators
\begin{align*}
 A_1=\mathcal{L},\quad A_2=2x+1/2,
\end{align*}
realize the relations \eqref{BI-Algebra} with values of the structure constants depending on the parameters $\rho_1$, $\rho_1$, $r_1$, $r_2$. Hence the Bannai--Ito algebra \eqref{BI-Algebra} encodes, inter alia, the bispectral properties \eqref{BI-Recurrence} and \eqref{BI-Eigenvalue} of the Bannai--Ito polynomials. Let us mention that since its introduction, the BI algebra has appeared in several instances, notably in connection with generalizations of harmonic \cite{2015_Genest&Vinet&Zhedanov_CommMathPhys} and Clifford \cite{2015_DeBie&Genest&Vinet_ArXiv_1501.03108} analysis involving Dunkl operators, and also as a symmetry algebra of superintegrable systems \cite{2014_Genest&Vinet&Zhedanov_JPhysA_47_205202}; see \cite{2015_DeBie&Genest&Tsujimoto&Vinet&Zhedanov} for an overview.

It was recently determined that the Bannai--Ito polynomials serve as Racah coefficients in the direct product of three unitary irreducible representations (UIRs) of the algebra $sl_{-1}(2)$ \cite{2014_Genest&Vinet&Zhedanov_ProcAmMathSoc_142_1545}. This algebra, introduced in \cite{2011_Tsujimoto&Vinet&Zhedanov_SIGMA_7_93}, is closely related to $\mathfrak{osp}(1|2)$ and its UIRs are associated to the one-dimensional para-Bose oscillator \cite{1980_Mukunda&Sudarshan&Sharma&Mehta_JMathPhys_21_2386}. The identification of the Bannai--Ito polynomials as Racah coefficients in \cite{2014_Genest&Vinet&Zhedanov_ProcAmMathSoc_142_1545} followed from the observation that the intermediate Casimir operators entering the Racah problem for $sl_{-1}(2)$ realize the Bannai--Ito algebra.

In this paper we consider the quantum superalgebra $\mathfrak{osp}_q(1|2)$; the Racah coefficients arising in the tensor product of three of its UIRs are calculated explicitly in terms of basic orthogonal polynomials that tend to the Bannai--Ito polynomials in the limit $q\rightarrow 1$. We give some of the properties of these $q$-deformed Bannai--Ito polynomials and discuss their relationship with the $q$-Racah and Askey-Wilson polynomials. The paper is divided as follows.

In Section two, the definition of the $\mathfrak{osp}_q(1|2)$ algebra is recalled and its extension by the grade involution is defined. UIRs of this extended $\mathfrak{osp}_{q}(1|2)$ and their Bargmann realizations are presented. In Section 3, the coproduct for $\mathfrak{osp}_q(1|2)$ is used to posit the Racah problem and the intermediate Casimir operators are introduced. It is shown that these operators realize a $q$-analog of the BI algebra. The finite-dimensional irreducible representations of this $q$-version of the BI algebra are constructed. These lead to an explicit expression of the Racah coefficients of $\mathfrak{osp}_q(1|2)$ in terms of $p$-Racah polynomials with base $p=-q$ which tend to the BI polynomials in the $q\rightarrow 1$ limit. In the fourth section, the $q$-analogs of the BI polynomials are defined independently from the Racah problem and their bispectral properties (recurrence relation and eigenvalue equation) are given explicitly; their relation with the Askey-Wilson polynomials is also detailed. In Section 5, the $q\rightarrow 1$ limit of the results is discussed. We conclude with an outlook.

\section{The quantum superalgebra $\mathfrak{osp}_q(1|2)$}
In this section, the definition of the quantum superalgebra $\mathfrak{osp}_q(1|2)$ is recalled and its extension by the grade involution is presented. The Hopf structure of this extended $\mathfrak{osp}_{q}(1|2)$ is described. UIRs of this algebra are constructed and their Bargmann realizations are provided.

\subsection{Definition and Casimir operator}
Let $q$ be a real number with $0<q<1$. The quantum superalgebra $\mathfrak{osp}_{q}(1|2)$ is the algebra presented in terms of one even generator $A_0$ and two odd generators $A_{\pm}$ obeying the commutation relations \cite{1989_Kulish&Reshetikhin_LettMathPhys_18_143}
\begin{align}
\label{OSP-1}
 [A_0,A_{\pm}]=\pm A_{\pm},\quad \{A_{+},A_{-}\}=[2A_0]_{q^{1/2}},
\end{align}
where $[x,y]=xy-yx$ is the commutator and where $[n]$ is the $q$-number
\begin{align*}
 [n]_{q}=\frac{q^{n}-q^{-n}}{q-q^{-1}}.
\end{align*}
 The abstract $\mathbb{Z}_2$-grading of the algebra \eqref{OSP-1} can be concretely realized by appending the grade involution $P$ to the set of generators and declaring that  the even and odd generators respectively commute and anticommute with $P$. The quantum superalgebra $\mathfrak{osp}_{q}(1|2)$ can hence be introduced as the algebra with generators $A_0$, $A_{\pm}$ and involution $P$ satisfying the commutation relations
 \begin{subequations}
 \label{OSPq}
 \begin{align}
 \label{First}
  [A_0,P]=0,\quad \{A_{\pm},P\}=0,\quad [A_0,A_{\pm}]=\pm A_{\pm},\quad \{A_{+},A_{-}\}=[2A_0]_{q^{1/2}},
 \end{align}
with $P^2=1$. It is convenient to define the operators
\begin{align*}
 K=q^{A_0/2},\quad K^{-1}=q^{-A_0/2}.
\end{align*}
In terms of these operators, the relations \eqref{First} read
\begin{gather}
\label{Second}
\begin{aligned}
 &K A_{+} K^{-1}=q^{1/2} A_{+},\quad K A_{-}K^{-1}=q^{-1/2} A_{-},\quad K K^{-1}=1,
 \\
 &[K,P]=0,\quad[K^{-1},P]=0,\quad \{A_{\pm},P\}=0,\quad \{A_{+},A_{-}\}=\frac{K^2-K^{-2}}{q^{1/2}-q^{-1/2}}.
 \end{aligned}
\end{gather}
\end{subequations}
We shall use both \eqref{First} and \eqref{Second}. The Casimir operator of $\mathfrak{osp}_q(1|2)$ reads
\begin{align}
\label{OSPq-Casimir}
 Q=\left[A_{+}A_{-}-\frac{q^{-1/2}K^2-q^{1/2}K^{-2}}{q-q^{-1}}\right]P.
\end{align}
It is easily verified that $Q$ commutes with all generators of \eqref{OSPq}. In \eqref{OSPq-Casimir}, the expression in the square bracket corresponds to the so-called sCasimir operator of $\mathfrak{osp}_{q}(1|2)$, which commutes with $A_0$ and anticommutes with $A_{\pm}$ \cite{1995_Lesniewski_JMathPhys_36_1457}.

\subsection{Hopf algebraic structure} The algebra \eqref{OSPq} can be endowed with a Hopf structure. Define the coproduct $\Delta: \mathfrak{osp}_q(1|2)\rightarrow\mathfrak{osp}_q(1|2)\otimes \mathfrak{osp}_q(1|2)$ as 
\begin{align}
\label{Coprod}
\Delta(A_{\pm})=A_{\pm}\otimes KP+K^{-1}\otimes A_{\pm}\quad \Delta(K)=K\otimes K,\quad \Delta(P)=P\otimes P,
\end{align}
the counit $\epsilon:\mathfrak{osp}_q(1|2)\rightarrow \mathbb{C}$ as
\begin{align}
 \label{Counit}
 \epsilon(P)=1,\quad \epsilon(K)=1,\quad& \epsilon(A_{\pm})=0,
\end{align}
and the coinverse $\sigma:\mathfrak{osp}_q(1|2)\rightarrow \mathfrak{osp}_q(1|2)$ by
\begin{align}
 \label{Coinverse}
 \sigma(P)=P,\quad \sigma(K)=K^{-1},\quad \sigma(A_{\pm})=q^{\pm 1/2} PA_{\pm}.
\end{align}
It is straightforward to verify that with \eqref{Coprod}, \eqref{Counit} and \eqref{Coinverse}, \eqref{OSPq} indeed has a Hopf algebraic structure. The conditions on $\Delta$, $\epsilon$ and $\sigma$ are well known; they can be found, for example, in Chap. 4 of \cite{2011_Underwood}.  The coproduct given in \eqref{Coprod} is not cocommutative since $\sigma \Delta\neq \Delta$, where $\sigma (a\otimes b)=b\otimes a$ is the flip automorphism. The alternative coproduct $\widetilde{\Delta}=\sigma \Delta$ and coinverse $\widetilde{S}=S^{-1}$ can be used to define another Hopf algebraic structure for \eqref{OSPq}; we shall not consider it here.

\begin{remark}
The coproduct \eqref{Coprod} appears different from the one presented in \cite{1989_Kulish&Reshetikhin_LettMathPhys_18_143}, as it explicitly involves the grade involution $P$. The two coproducts are however equivalent.  For elements in $\mathfrak{osp}_q(1|2)\otimes \mathfrak{osp}_q(1|2)$ a graded product law of the form $(a\otimes b)(c\otimes d)=(-1)^{p(b)}(-1)^{p(c)}(ac\otimes bd)$, where $p(x)$ gives the parity of $x$, was used in \cite{1989_Kulish&Reshetikhin_LettMathPhys_18_143} whereas the standard product rule $(a\otimes b)(c\otimes d)=ac\otimes bd$ is used here.
\end{remark}
\subsection{Unitary irreducible $\mathfrak{osp}_q(1|2)$-modules} Let $\epsilon$, $\mu$ be real numbers such that $\mu>0$, $\epsilon=\pm 1$ and let $W^{(\epsilon,\mu)}$ denote the infinite-dimensional vector space  spanned by the orthonormal basis vectors $\ket{\epsilon,\mu;n}$ where $n$ is a non-negative integer. The basis vectors satisfy 
\begin{align*}
 \braket{\epsilon,\mu;n'}{\epsilon,\mu;n}=\delta_{nn'},
\end{align*}
where $\delta$ is the Kronecker delta. Consider the $\mathfrak{osp}_q(1|2)$ actions
\begin{alignat}{2}
\label{Actions}
\begin{aligned}
 A_0\,\ket{\epsilon,\mu;n}&=(n+\mu+1/2)\,\ket{\epsilon,\mu;n},\quad & P\,\ket{\epsilon,\mu;n}&=\epsilon\, (-1)^{n}\,\ket{\epsilon,\mu;n},
 \\
 A_{+}\ket{\epsilon,\mu;n}&=\sqrt{\sigma_{n+1}}\,\ket{\epsilon,\mu;n+1},\quad & A_{-}\,\ket{\epsilon,\mu;n}&=\sqrt{\sigma_n}\,\ket{\epsilon,\mu;n-1},
 \end{aligned}
\end{alignat}
where $\sigma_n$ is of the form
\begin{align*}
 \sigma_n=[n+\mu]_{q}-(-1)^{n}[\mu]_{q},\quad n=0,1,2,\ldots.
\end{align*}
The vector space $W^{(\epsilon,\mu)}$ endowed with the actions \eqref{Actions}  forms a unitary irreducible $\mathfrak{osp}_q(1|2)$-module. Indeed, it is verified that the actions \eqref{Actions} comply with \eqref{OSPq}. The irreducibility follows from the fact that $\sigma_n>0$ for $n\geqslant 1$. The module $W^{(\epsilon,\mu)}$ is unitary, as it is realizes the $\star$-conditions
\begin{align}
\label{Star-Conditions}
 A_0^{\dagger}=A_0,\quad P^{\dagger}=P,\quad A_{\pm}^{\dagger}=A_{\mp}.
\end{align}
The representation space $W^{(\epsilon,\mu)}$ can be identified with the state space of the one-dimensional $q$-deformed parabosonic oscillator \cite{1990_Floreanini&Vinet_JPhysA_23_L1019}. On $W^{(\epsilon,\mu)}$, the Casimir operator \eqref{OSPq-Casimir} has the action
\begin{align}
\label{Multiple}
 Q\,\ket{\epsilon,\mu;n}=-\epsilon\,[\mu]_{q}\,\ket{\epsilon,\mu;n}.
\end{align}
The modules $W^{(\epsilon,\mu)}$ have a Bargmann realization on functions of argument $z$. In this realization, the basis vectors $\ket{\epsilon,\mu;n}\equiv e_n^{(\epsilon,\mu)}(z)$ have the expression
\begin{align*}
 e_{n}^{(\epsilon,\mu)}(z)=\frac{z^{n}}{\sqrt{\sigma_1 \sigma_2\cdots \sigma_n}},\qquad n=0,1,2,\ldots,
\end{align*}
and the $\mathfrak{osp}_q(1|2)$ generators take the form
\begin{gather}
\nonumber
 A_0(z)=z\partial_{z}+\mu+1/2,\quad K(z)=q^{(\mu+1/2)/2} T_{q}^{1/2},
 \\
 \label{Bargann-Realization}
 P(z)=\epsilon R_z,\qquad A_{+}(z)=z,
 \\
 \nonumber
 A_{-}(z)=q^{\mu}\frac{(T_{q}-R_z)}{(q-q^{-1})z}-q^{-\mu}\frac{(T_{q}^{-1}-R_z)}{(q-q^{-1})z},
\end{gather}
where $T_{q}^{h}f(z)=f(q^{h}z)$ and $R_z f(z)=f(-z)$.
\section{The Racah problem}
In this section, the Racah problem for $\mathfrak{osp}_q(1|2)$ is considered. The intermediate Casimir operators are defined and are seen to generate a $q$-analog of the Bannai--Ito algebra. The eigenvalues of the intermediate Casimirs are derived and the corresponding representations of the $q$-extended Bannai--Ito algebra are constructed. The explicit expression of the Racah coefficients for $\mathfrak{osp}_{q}(1|2)$ in terms of orthogonal polynomials is given.
\subsection{Outline the problem} The coproduct of $\mathfrak{osp}_{q}(1|2)$ allows to construct tensor product representations. Consider the $\mathfrak{osp}_{q}(1|2)$-module defined by
\begin{align}
\label{W}
 W=W^{(\epsilon_1,\mu_1)}\otimes W^{(\epsilon_2,\mu_2)}\otimes W^{(\epsilon_3,\mu_3)}.
\end{align}
The action of any generator $X$ on $W$ is prescribed by $(1\otimes \Delta)\Delta(X)$ or equivalently by $(\Delta\otimes 1)\Delta(X)$ since the coproduct is coassociative. When considering three-fold tensor product representations, three types of Casimir operators arise. There are three \emph{initial} Casimir operators $Q^{(1)}$, $Q^{(2)}$, $Q^{(3)}$ defined by
\begin{align}
\label{Initial-Casimir}
 Q^{(1)}=Q\otimes 1\otimes 1,\quad Q^{(2)}=1\otimes Q\otimes 1,\quad Q^{(3)}=1\otimes 1\otimes Q,
\end{align}
which are associated to each components of the tensor product \eqref{W}. On $W$, each initial Casimir operator $Q^{(i)}$ acts as a multiple of the identity. In view of \eqref{Multiple}, this multiple denoted by  $\tau_i$ is given by
\begin{align}
\label{Values}
\tau_i=-\epsilon_i\,[\mu_i]_{q},\qquad i=1,2,3.
\end{align}
There are two \emph{intermediate} Casimir operators $Q^{(12)}$, $Q^{(23)}$ defined by
\begin{align}
\label{Intermediate-Casimir}
 Q^{(12)}=\Delta(Q)\otimes 1,\quad Q^{(23)}=1\otimes \Delta(Q),
\end{align}
which are associated to $W^{(\epsilon_1,\mu_1)}\otimes W^{(\epsilon_2,\mu_2)}$ and $W^{(\epsilon_2,\mu_2)}\otimes W^{(\epsilon_3,\mu_3)}$, respectively. A direct calculation using \eqref{OSPq-Casimir} and \eqref{Coprod} shows that $\Delta(Q)$ has the expression
\begin{multline*}
 \Delta(Q)=q^{1/2}\left(A_{-}K^{-1}P\otimes A_{+}K\right)-q^{-1/2}\left( A_{+}K^{-1}P\otimes A_{-}K\right)
 \\
 -[1/2]_{q}\,K^{-2}P\otimes K^{2}P+Q\otimes K^2P+K^{-2}P\otimes Q.
\end{multline*}
Finally, there is the \emph{total} Casimir operator $\mathcal{Q}$ defined by
\begin{align*}
 \mathcal{Q}=(1\otimes \Delta)\Delta(Q)=(\Delta\otimes 1)\Delta(Q),
\end{align*}
which is associated to the whole module $W$. The total Casimir operator reads
\begin{multline}
\label{Total-Casimir}
\mathcal{Q}=q^{1/2}\left(A_{-}K^{-1}P\otimes1\otimes A_{+}K\right)-q^{-1/2}\left(A_{+}K^{-1}P\otimes 1\otimes A_{-}K\right)
\\
-K^{-2}P\otimes Q\otimes K^2P+\Delta(Q)\otimes K^2P+K^{-2}P\otimes \Delta(Q).
\end{multline}
The operators $Q^{(12)}$ and $Q^{(23)}$ both commute with $\mathcal{Q}$, but they do not commute with one another.  Moreover, the operators $Q^{(12)}$, $Q^{(23)}$ and $\mathcal{Q}$ all commute by construction with the operator $E$ which reads
\begin{align*}
 E=(1\otimes \Delta)\Delta (A_0)=A_0\otimes 1\otimes 1+1\otimes A_0\otimes 1+1\otimes 1\otimes A_0.
\end{align*}
Each of $\{Q^{(12)},\mathcal{Q}, E\}$ and $\{Q^{(23)}, \mathcal{Q},E\}$ forms a complete set of self-adjoint commuting operators with respect to $W$.

We introduce two distinct bases for $W$ associated to the two complete sets of commuting operators exhibited above. The first one consists of the orthonormal basis vectors $\ket{m;\tau_{12};\tau}$ defined by the eigenvalue equations
\begin{equation}
\label{Basis-12}
\begin{gathered}
 Q^{(12)}\ket{m;\tau_{12};\tau}=\tau_{12}\ket{m;\tau_{12};\tau},\quad \mathcal{Q}\ket{m;\tau_{12};\tau}=\tau \ket{m;\tau_{12};\tau},\quad 
 \\
 E \ket{m;\tau_{12};\tau}=m\ket{m;\tau_{12};\tau}.
 \end{gathered}
\end{equation}
The second one consists of the orthonormal basis vectors $\ket{m;\tau_{23};\tau}$ defined by the eigenvalue equations
\begin{equation}
\label{Basis-23}
\begin{gathered}
 Q^{(23)}\ket{m;\tau_{23};\tau}=\tau_{23}\ket{m;\tau_{23};\tau},\quad \mathcal{Q}\ket{m;\tau_{23};\tau}=\tau \ket{m;\tau_{23};\tau},\quad 
 \\
 E \ket{m;\tau_{23};\tau}=m\ket{m;\tau_{23};\tau}.
 \end{gathered}
\end{equation}
The Racah problem consist in the determination of the \emph{Racah coefficients}, which are the transition coefficients between these two orthonormal bases. Such coefficients are easily shown to be independent of $m$ \cite{2003_VanderJeugt_OrthPoly&SpecFuns_25}. We hence write
\begin{align}
\label{Racah-Coef}
 \begin{bmatrix}
  \tau_1  & \tau_2 & \tau_3
  \\
  \tau_{12} & \tau_{23} & \tau
 \end{bmatrix}
 =\braket{m;\tau_{12};\tau}{m;\tau_{23};\tau},
\end{align}
and refer to the left-hand side of \eqref{Racah-Coef} as the Racah coefficients for $\mathfrak{osp}_q(1|2)$. For more details on the Racah problem for $\mathfrak{sl}(2)$ and $\mathfrak{sl}_q(2)$, one can consult \cite{2011_Groenevelt_SIGMA_7_77, 1991_Vilenkin&Klimyk}.
\subsection{Main observation: $q$-deformation of the Bannai--Ito algebra} The properties of the Racah coefficients are encoded in the algebraic interplay between the intermediate and total Casimir operators. A fruitful approach is therefore to investigate the commutation relations that these operators satisfy \cite{2014_Genest&Vinet&Zhedanov_ProcAmMathSoc_142_1545,1988_Granovskii&Zhedanov_JETP_94_49}. Introduce the operators $I_3$ and $I_1$ defined as
\begin{align}
\label{Generator-Def}
 I_1=- Q^{(23)},\qquad I_3=-Q^{(12)}.
\end{align}
Let $\{A,B\}_{q}$ denote the ``$q$-anticommutator'' 
\begin{align*}
 \{A,B\}_{q}=q^{1/2} AB+q^{-1/2} BA,
\end{align*}
and introduce the operator $I_2$ through the relation
\begin{align*}
\{I_3,I_1\}_{q}\equiv I_2+(q^{1/2}+q^{-1/2})\left[Q^{(3)}Q^{(1)}+Q^{(2)}\mathcal{Q}\right].
\end{align*}
An involved but direct calculation shows that these operators satisfy the relations
\begin{align*}
\{I_i,I_j\}_{q}=I_k+(q^{1/2}+q^{-1/2})\left[Q^{(i)}Q^{(j)}+Q^{(k)}\mathcal{Q}\right],
\end{align*}
where $(ijk)$ is an even permutation of $\{1,2,3\}$. It follows that the bases \eqref{Basis-12} and \eqref{Basis-23} that enter the Racah problem support representations of the algebra
\begin{align}
\label{q-BI}
 \{I_{i},I_j\}_q=I_k+\iota_k,\qquad \iota_{k}=(q^{1/2}+q^{-1/2})(\tau\tau_k+\tau_i\tau_j),
\end{align}
where $(ijk)$ is an even permutation of $\{1,2,3\}$ and where $\tau_i$ is given by \eqref{Values}; the values of $\tau$ that can occur remain to be evaluated. Since $I_3$ and $I_1$ are proportional to $Q^{(12)}$ and $Q^{(23)}$, the Racah coefficients \eqref{Racah-Coef} coincide with the transition coefficients between the eigenbases of $I_3$ and $I_1$ in the appropriate representations of \eqref{q-BI}, which will be studied below. The operator 
\begin{multline}
\label{Casimir-q-BI}
 C=(q^{-1/2}-q^{3/2})I_1I_2I_3+q I_1^2+q^{-1}I_2^2+q I_3^2
 \\
 -(1-q)\,\iota_1\, I_1-(1-q^{-1})\, \iota_2\,  I_2-(1-q)\, \iota_3\,  I_3,
\end{multline}
can be seen to commute with $I_1$, $I_2$ and $I_3$. After considerable algebra, one finds that on the bases \eqref{Basis-12} and \eqref{Basis-23}, the operator $C$ takes the value
\begin{align}
\label{Casimir-Value}
 C=-(q-q^{-1})^2 \tau_1\tau_2\tau_3\tau+\tau_1^2+\tau_2^2+\tau_3^2+\tau^2-q/(1+q)^2.
\end{align}
The algebra \eqref{q-BI} stands as a $q$-deformation of the Bannai--Ito algebra with $C$ as its Casimir operator. 

Let us note that the algebra \eqref{q-BI} can be presented in terms of only two generators. Eliminating $I_2$ from \eqref{q-BI}, one finds that $I_1$ and $I_3$ satisfy
\begin{subequations}
\label{Rels}
\begin{align}
\label{Rel-A}
I_1^2I_3+(q+q^{-1})I_1I_3I_1+I_3I_1^2&=I_3+(q^{1/2}+q^{-1/2})\,\iota_2\,I_1+\iota_3,
\\
\label{Rel-B}
I_3^2I_1+(q+q^{-1})I_3I_1I_3+I_1I_3^2&=I_1+(q^{1/2}+q^{-1/2})\,\iota_2\,I_3+\iota_1.
\end{align}
\end{subequations}
\begin{remark}
The algebra \eqref{q-BI} can be obtained from the Zhedanov algebra \cite{1991_Zhedanov_TheorMathPhys_89_1146} by the formal substitution $q\rightarrow -q$ and scaling of the generators. The Zhedanov algebra was also studied by Koornwinder \cite{2007_Koornwinder_SIGMA_3_63} and Terwilliger \cite{2004_Terwilliger&Vidunas_JAlgebraAppl_3_411}.
\end{remark}
\subsection{Spectra of the Casimir operators} To investigate the Racah problem, we need to identify which representations of \eqref{q-BI} arise; this is done by determining the eigenvalues of the intermediate and total Casimir operators of $\mathfrak{osp}_q(1|2)$.

The eigenvalues of the intermediate Casimir operator $Q^{(12)}$, and hence those of $I_3$, are associated to the decomposition of the two-fold tensor product module $\widetilde{W}=W^{(\epsilon_1,\mu_1)}\otimes W^{(\epsilon_2,\mu_2)}$ in irreducible components. As a vector space, $\widetilde{W}$ has the direct sum decomposition
\begin{align*}
 \widetilde{W}=\bigoplus_{n=0}^{\infty} U_{n},
\end{align*}
where each $U_{n}$ is an eigenspace of $\Delta(A_0)$ with eigenvalue $n+\mu_1+\mu_2+1$. It is seen that $U_{n}$ is $(N+1)-$dimensional, as it is spanned by vectors $\ket{\epsilon_1,\mu_1;n_1}\otimes \ket{\epsilon_2,\mu_2;n_2}$ such that $n_1+n_2=n$. Since $\Delta(A_0)$ and $ \Delta(Q)$ commute, $U_{n}$ is stabilized by $\Delta(Q)$.
\begin{lemma}
 The eigenvalues of $\Delta(Q)$ on $U_n$ have the expression
 \begin{align}
 \label{Eigen}
  \vartheta_k=(-1)^{k+1}\epsilon_1\epsilon_2\,[k+\mu_1+\mu_2+1/2]_{q},\quad k=0,1,\ldots, n.
 \end{align}
\end{lemma}
\begin{proof}
 By induction on $n$. The case $n=0$ is verified directly by applying  $\Delta(Q)$ on the single basis vector $\ket{\epsilon_1,\mu_1;0}\otimes \ket{\epsilon_2,\mu_2;0}$ of $U_0$. Suppose that \eqref{Eigen} holds at level $n-1$ and let $v_k\in U_{n-1}$ for $k=1,\ldots,n-1$ denote the eigenvectors of $\Delta(Q)$ with eigenvalues \eqref{Eigen}. It is directly seen from the relations \eqref{OSPq} that the vectors $\Delta(A_{+}) v_{k}$ are in $U_n$ and that they are eigenvectors of $\Delta(Q)$ with the same eigenvalues. Consider the vector $w\in U_n$ such that $\Delta(A_{-})w=0$; such a vector is easily constructed in the direct product basis by solving a two-term recurrence relation. It follows from \eqref{Coprod} that $w$ is an eigenvector of $\Delta(P)$ with eigenvalue $(-1)^{n}\epsilon_1\epsilon_2$. A calculation shows that $w$ is an eigenvector of $\Delta(Q)$ with eigenvalue $\vartheta_n$. Hence the eigenvalues of $\Delta(Q)$ on $U_n$ are $\{\vartheta_0,\vartheta_1,\ldots, \vartheta_{n-1}\}\cup \{\vartheta_n\}$.
\end{proof}
It follows from the above lemma that one has the direct sum decomposition
\begin{align}
\label{Decomposition}
 W^{(\epsilon_i,\mu_i)}\otimes W^{(\epsilon_j,\mu_j)}=\bigoplus_{k=0}^{\infty} W^{(\epsilon_{ij}(k),\,\mu_{ij}(k))},
\end{align}
where
\begin{align}
\label{Numbers2}
 \epsilon_{ij}(k)=(-1)^{k}\epsilon_1\epsilon_2,\qquad \mu_{ij}(k)=k+\mu_1+\mu_2+1/2.
\end{align}
Upon using the decomposition \eqref{Decomposition} twice, one finds that the decomposition of the $\mathfrak{osp}_{q}(1|2)$-module $W$ in irreducible components has the form
\begin{align*}
 W=\bigoplus_{N=0}^{\infty}m_{N}\,W^{(\epsilon_{N}, \mu_{N})},
\end{align*}
where the multiplicity is $m_{N}=N+1$ and where 
\begin{align}
\label{Numbers}
 \epsilon_{N}=(-1)^{N}\epsilon_1\epsilon_2\epsilon_3,\qquad \mu_{N}=N+\mu_1+\mu_2+\mu_3+1.
\end{align}

It follows from the above discussion that the eigenvalues $\tau$ of the total Casimir operator $\mathcal{Q}$ are parametrized by the non-negative integer $N$ and read
\begin{align}
\label{TAUN}
 \tau\rightarrow \tau_{N}=-\epsilon_{N}\,[\mu_{N}]_{q},\qquad N=0,1,\ldots
\end{align}
where $\epsilon_{N}$ and $\mu_{N}$ are given by \eqref{Numbers}. The eigenvalues  $\tau_{12}$ and $\tau_{23}$ of the intermediate Casimir operators $Q^{(12)}$, $Q^{(23)}$ are respectively parametrized by the non-negative integers $n$, $s$ and read
\begin{alignat}{2}
\label{Tau_Inter}
\begin{aligned}
 \tau_{12}\rightarrow \tau_{12}(n)&=-\epsilon_{12}(n)\,[\mu_{12}(n)]_{q},\quad& n&=0,1,\ldots, N,
 \\
 \tau_{23}\rightarrow \tau_{23}(s)&=-\epsilon_{23}(s)\,[\mu_{23}(s)]_{q},\quad& s&=0,1,\ldots, N,
 \end{aligned}
\end{alignat}
where $\epsilon_{ij}(k)$ and $\mu_{ij}(k)$ are given by \eqref{Numbers2}. We can thus write the Racah coefficients for $\mathfrak{osp}_q(1|2)$ as 
\begin{align}
\label{Racah-Coef-2}
\left[\begin{matrix}
  \tau_1 & \tau_2 & \tau_3
  \\
  \tau_{12}(n) & \tau_{23}(s) & \tau_{N}
 \end{matrix}\right],\qquad n,s\in \{0,1,\ldots, N\},\qquad N=0,1,2,\ldots.
 \end{align}
These coefficients coincide with the interbasis expansion coefficients between the eigenbases of $I_1$ and $I_3$ in the $(N+1)$-dimensional representations of the algebra \eqref{q-BI} with Casimir value \eqref{Casimir-Value}.
\subsection{Representations} We construct the matrix elements of $I_1$ in the eigenbasis of $I_3$. In view of \eqref{Generator-Def}, \eqref{Numbers2}, \eqref{Tau_Inter}, the eigenvectors of $I_3$ denoted by $\ket{N;n}$ satisfy
\begin{align}
\label{Eigen2}
 I_3 \ket{N;n}=\lambda_n\, \ket{N;n},\quad n=0,1,\ldots,N,
\end{align}
where $\lambda_n=-\tau_{12}(n)$. The action of the operator $I_1$ on this basis can be written as
\begin{align}
\label{Matrix2}
 I_1\ket{N;n}=\sum_{k=0}^{N}A_{kn}\ket{N;k},
\end{align}
where $A_{kn}$ are the matrix elements of $I_1$. In view of \eqref{Eigen2}, \eqref{Matrix2} and since the basis vectors are linearly independent, the relation  \eqref{Rel-B} is equivalent to
\begin{align}
 \label{Eq-1}
 A_{kn}\left[\lambda_n^2+(q+q^{-1})\lambda_n\lambda_k+\lambda_k^2-1\right]=\delta_{kn}[(q^{1/2}+q^{-1/2}) \iota_2+\iota_1].
\end{align}
For $k\neq n$, the left-hand side of \eqref{Eq-1} must vanish. It is seen from the expression of the eigenvalues $\lambda_n$ that $A_{kn}$ can be non-zero only when $k=n\pm 1$ or $k=n$. As a result, the matrix representing $I_1$ in the $I_3$ eigenbasis is tridiagonal, i.e.
\begin{align}
\label{Tridiag}
 I_1 \ket{N;n}=U_{n+1}\ket{N;n+1}+ V_{n}\ket{N;n}+U_{n-1}\ket{N;n-1},
\end{align}
where by definition $U_0=0$, $U_{N+1}=0$ and where we have used the fact that $I_1$ is self-adjoint. When $n=k$, equation \eqref{Eq-1} gives the following expression for $V_{n}$:
\begin{align}
\label{Vn}
 V_{n}=\frac{\iota_1+(q^{1/2}+q^{-1/2})\iota_2 \lambda_n}{\lambda_n^2(2+q+q^{-1})-1}.
\end{align}
If one acts with the relation \eqref{Rel-A} on $\ket{N;n}$ and gathers all terms proportional to $\ket{N;n}$, one finds that $U_{n}^2$ satisfies the two-term recurrence relation
\begin{multline}
\label{Recurrence-A}
 (2\lambda_n+(q+q^{-1})\lambda_{n+1})\,U_{n+1}^2+(2\lambda_n+(q+q^{-1})\lambda_n) V_{n}^2
 \\
 +(2\lambda_n+(q+q^{-1})\lambda_{n-1})\,U_{n}^2=\lambda_n+(q^{1/2}+q^{-1/2})\iota_2 V_n+\iota_3.
\end{multline}
The solution to \eqref{Recurrence-A} can be presented as follows. Let $a$, $b$, $c$ and $d$ be defined as
\begin{alignat}{2}
\label{parameters}
\begin{aligned}
 a&=\epsilon_2\epsilon_3\,q^{\mu_2+\mu_3+1/2},\quad& b&=-\epsilon_1\epsilon_{N}\,q^{\mu_1-\mu_{N}+1/2},
 \\
 c&=-\epsilon_{1}\epsilon_{N}\,q^{\mu_1+\mu_{N}+1/2},\quad & d&=\epsilon_2\epsilon_3\,q^{\mu_2-\mu_3+1/2},
 \end{aligned}
\end{alignat}
and let $p=-q$. Up to an inessential phase factor, one has
\begin{align}
\label{Un}
 U_{n}=(q-q^{-1})^{-1}\,\sqrt{A_{n-1}C_n},
\end{align}
where $A_n$ and $C_n$ read
\begin{align}
\label{Recurrence-Coef}
\begin{aligned}
 A_{n}&=-\frac{(1+a b p^{n})(1-a c p^{n})(1-a d p^{n})(1-abcd p^{n-1})}{a(1-abcd p^{2n-1})(1-abcd p^{2n})},
 \\
 C_{n}&=\frac{a(1-p^{n})(1-bcp^{n-1})(1-bd p^{n-1})(1+cdp^{n-1})}{(1-abcd p^{2n-2})(1-abcd p^{2n-1})}.
 \end{aligned}
\end{align}
The coefficients $V_n$ given in \eqref{Vn} can be written as
\begin{align}
\label{Vn-2}
 V_{n}=(q-q^{-1})^{-1}\left[a-a^{-1}-A_n-C_n\right].
\end{align}
With $U_n$ and $V_n$ as in \eqref{Un} and \eqref{Vn-2}, the actions \eqref{Eigen2} and \eqref{Tridiag} define $(N+1)$-dimensional representations of \eqref{Rels} with value \eqref{Casimir-Value} of the Casimir operator \eqref{Casimir-q-BI}. Since $U_{n}\neq 0$ for $1 \leqslant n\leqslant N$, these representations are irreducible.

The matrix elements of the generators $I_1$, $I_3$ in the eigenbasis of $I_1$ are easily obtained. One observes that the relations \eqref{q-BI} and Casimir value \eqref{Casimir-Value} are all invariant under simultaneous cyclic permutations of the generators $I_i$ and representation parameters $\mu_i$ and $\epsilon_i$. As a result, the matrix elements of $I_1$, $I_3$ in the eigenbasis  $\{\ket{N;s}\}_{s=0}^{N}$ of $I_1$ are of the form
\begin{align}
\label{s-Basis}
\begin{aligned}
 I_1 \ket{N;s}&=\widetilde{\lambda}_s\ket{N;s},\qquad s=0,1,\ldots,N,
 \\
 I_3 \ket{N;s}&=\widetilde{U}_{s+1}\ket{N;s+1}+\widetilde{V}_{s}\ket{N;s}+\widetilde{U}_{s}\ket{N;s-1},
 \end{aligned}
\end{align}
where $\widetilde{\lambda}_s$, $\widetilde{U}_s$ and $\widetilde{V}_s$ are obtained from \eqref{Eigen2}, \eqref{Un} and \eqref{Vn-2} by applying the permutations $(\mu_1,\mu_2,\mu_3)\rightarrow (\mu_2,\mu_3,\mu_1)$ and $(\epsilon_1,\epsilon_2,\epsilon_3)\rightarrow(\epsilon_2,\epsilon_3,\epsilon_1)$.
\subsection{The Racah coefficients of $\mathfrak{osp}_q(1|2)$ as basic orthogonal polynomials} As explained in Subsection 3.3, the Racah coefficients \eqref{Racah-Coef-2} of $\mathfrak{osp}_q(1|2)$ coincide with the overlap coefficients $\braket{N;s}{N;n}$. These coefficients can be cast in the form
\begin{align}
\label{Cast}
 \braket{N;s}{N;n}=\omega_{s}\,G_{n}(s),\quad \text{where}\quad \omega_{s}=\braket{N;s}{N;0}\quad \text{and}\quad G_{0}(s)\equiv 1.
 \end{align}
Upon considering $\Braket{N;s}{I_1}{N;n}$ together with \eqref{Tridiag} and \eqref{s-Basis}, one finds that $G_{n}(s)$ satisfies the three-term recurrence relation
\begin{multline}
\label{Recurrence-1}
 (-1)^{s}\,(aq^{s}-a^{-1}q^{-s})\,G_{n}(s)=
 \\
 \sqrt{A_{n}C_{n+1}} G_{n+1}(s)+[a-a^{-1}-A_{n}-C_{n}]G_{n}(s)+\sqrt{A_{n-1}C_{n}} G_{n-1}(s),
\end{multline}
where $A_{n}$, $C_{n}$ are given by \eqref{Recurrence-Coef} with the parameterization \eqref{parameters}. 

If one takes
\begin{align*}
 \widehat{G}_{n}(s)=(-1)^{n}a^{n}\sqrt{A_0\cdots A_{n-1}\;C_1\cdots C_{n}}\,G_{n}(s),
\end{align*}
one finds that $\widehat{G}_{n}(s)$ satisfies the normalized recurrence relation
\begin{multline}
\label{Recu-5}
(p^{-s}+q^{2\mu_2+2\mu_3} p^{s})\,\widehat{G}_{n}(s)\\
 =\widehat{G}_{n+1}(s)+[1+q^{2\mu_2+2\mu_3}\,p-\check{A}_{n}-\check{C}_{n}]\widehat{G}_{n}(s)+\check{A}_{n-1}\check{C}_{n}\widehat{G}_{n-1}(s),
 \end{multline}
where  $p=-q$ and where
\begin{align*}
 \check{A}_{n}=-a\,A_{n},\qquad \check{C}_{n}=-a\,C_{n}.
\end{align*}
The recurrence relation \eqref{Recu-5} coincides with the normalized recurrence relation for the $p$-Racah polynomials $R_{n}(\mu(s);\alpha,\beta,\gamma,\delta\,\rvert\, p)$ of degree $n$ in the variable $\mu(s)=p^{-s}+\gamma \delta p^{s+1}$ \cite{2010_Koekoek_&Lesky&Swarttouw}. In consequence, the functions $G_{n}(s)$ appearing in the coefficients \eqref{Cast} are proportional to the $p$-Racah polynomials
\begin{align}
\label{Racah}
 R_{n}(\mu(s);\alpha,\beta,\gamma,\delta\,\rvert\,p)=
 {}_4\varphi_{3}
 \left(
 \genfrac{}{}{0pt}{}{p^{-n}, \alpha\beta p^{n+1},p^{-s},\gamma\delta p^{s+1}}{\alpha p, \beta\delta p, \gamma p}
 ;p,p
 \right),
\end{align}
where ${}_r\varphi_{s}$ is the generalized basic hypergeometric series \cite{2010_Koekoek_&Lesky&Swarttouw}
\begin{align*}
 {}_r\varphi_{s}
 \left(
 \genfrac{}{}{0pt}{}{a_1, \ldots,a_{r}}{b_1,\ldots,b_{s}}
 ;q,z
 \right)
 =
 \sum_{k=0}^{\infty}\frac{(a_1,\cdots a_r;q)_{k}}{(b_1,\cdots, b_{s};q)_{k}}(-1)^{(1+s-r)k} q^{(1+s-r)\binom{k}{2}}\;\frac{z^{k}}{(q;q)_{k}}.
\end{align*}
and where have used the standard notation:
\begin{align*}
 (a_1,a_2,\ldots,a_{k};q)_s=\prod_{i=1}^{k}(a_i;q)_{s},\qquad (a;q)_{s}=\prod_{k=1}^{s}(1-q^{k-1}a).
\end{align*}
Recall that $p=-q$. The relation between the parameters $\alpha,\beta,\gamma,\delta$ of the $p$-Racah polynomials and those appearing in \eqref{Recu-5} is
\begin{alignat}{2}
\label{Para}
\begin{aligned}
 \alpha&=-(-1)^{N}q^{\mu_1+\mu_2+\mu_3-\mu_{N}},\quad& \gamma&=-q^{2\mu_2} ,
 \\
 \beta&=-(-1)^{N}q^{\mu_1+\mu_2-\mu_3+\mu_{N}},\quad &\delta&=-q^{2\mu_3}.
 \end{aligned}
\end{alignat}
Using the expression \eqref{Numbers} for $\mu_{N}$, it is seen that one has $\alpha\,p=p^{-N}$, which is one of the admissible truncation condition for the $p$-Racah polynomials. 

The vectors $\ket{N;n}$ being orthonormal, one has the orthogonality relation
\begin{align}
\label{Ortho-1}
 \sum_{s=0}^{N}\braket{N;n'}{N;s}\braket{N;s}{N;n}=\sum_{s=0}^{N}\omega_{s}^2\,G_{n}(s)G_{n'}(s)=\delta_{nn'}.
\end{align}
Since the orthogonality weight for the $p$-Racah polynomials is unique, one concludes that $\omega_{s}$ in  \eqref{Cast} is the square root of the $p$-Racah weight function with parameters \eqref{Para}. The weight function $\Omega_{s}$ of the $p$-Racah polynomials reads \cite{2010_Koekoek_&Lesky&Swarttouw}
\begin{align*}
 \Omega_{s}(\alpha,\beta,\gamma,\delta;p)=\frac{(\alpha p, \beta\delta p,\gamma p, \gamma \delta p;p)_{s}}{(p,\alpha^{-1}\gamma \delta p, \beta^{-1}\gamma p,\delta p;p)_{s}}\frac{1-\gamma \delta p^{2s+1}}{(\alpha\beta p)^{s}(1-\gamma \delta p)},
\end{align*}
and the normalization coefficients $h_{n}$ are
\begin{multline*}
 h_{n}(\alpha,\beta,\gamma,\delta;p)=\frac{(\beta^{-1},\gamma \delta p^2;p)_{N}}{(\beta^{-1}\gamma p, \delta p; p)_{N}}\frac{(1-\beta p^{-N})(\gamma \delta p)^{n}}{(1-\beta p^{2n-N})}
 \\
 \times \frac{(p,\beta p, \beta \gamma^{-1}p^{-N},\delta^{-1}p^{-N};p)_{n}}{(\beta p^{-N},\beta \delta p,\gamma p,p^{-N};p)_{n}}.
\end{multline*}
The complete and explicit expression for the Racah coefficients of $\mathfrak{osp}_q(1|2)$ arising in the tensor product of three irreducible modules $W^{(\epsilon_i,\mu_i)}$ is thus
\begin{align*}
 \left[\begin{matrix}
  \tau_1 & \tau_2 & \tau_3
  \\
  \tau_{12}(n) & \tau_{23}(s) & \tau_{N}
 \end{matrix}\right]=(-1)^{n}\sqrt{\frac{\Omega_{s}(\alpha,\beta,\gamma,\delta;p)}{h_{n}(\alpha,\beta,\gamma,\delta;p)}} \;R_{n}(\mu(s);\alpha,\beta,\gamma,\delta;p),
\end{align*}
with the parametrization \eqref{Para} and $p=-q$. Let us remark that these Racah coefficients do not depend on the representation parameters $\epsilon_1,\epsilon_2,\epsilon_3$. 
\begin{remark}
 The Racah, or $6j$, coefficients of $\mathfrak{osp}_q(1|2)$ were also studied in \cite{1995_Minnaert&Mozrzymas_JMathPhys_36_907}. The authors considered different representations than the ones considered here. They focused in particular on finite-dimensional representations. The connection with orthogonal polynomials and the algebraic structure \eqref{q-BI} were not discussed.
\end{remark}
\begin{remark}
 The Clebsch-Gordan (CG) problem for $\mathfrak{osp}_q(1|2)$ arising in the tensor product of two irreducible representations was considered in \cite{1997_Chung&Kalnins&Miller_JPhysA_30_7147} and basic orthogonal polynomials with base $p=-q$ were seen to arise as CG coefficients.
\end{remark}
\section{$q$-analogs of the Bannai--Ito polynomials and Askey-Wilson polynomials with base $p=-q$}
In this section, the basic polynomials with basis $p=-q$ encountered above are presented independently from the Racah problem of $\mathfrak{osp}_{q}(1|2)$. Their relation with the Askey-Wilson polynomials with base $p=-q$ is discussed.

Consider the recurrence relation \eqref{Recurrence-1}, the recurrence coefficients \eqref{Recurrence-Coef} and the parametrization \eqref{parameters}. Defining $z=a\,p^{s}$, one is naturally led to introduce the polynomials $Q_{n}(x;a,b,c,d;q)\equiv Q_{n}(x)$ defined by the recurrence relation
\begin{align}
\label{Recurrence-2}
 (z-z^{-1})Q_{n}(x)=A_{n}Q_{n+1}(x)+[a-a^{-1}-A_{n}-C_{n}]Q_{n}(x)+C_{n}Q_{n-1}(x),
\end{align}
where $x=z-z^{-1}$ and where the recurrence coefficients read
\begin{align*}
\begin{aligned}
 A_{n}&=-\frac{(1+a b p^{n})(1-a c p^{n})(1-a d p^{n})(1-abcd p^{n-1})}{a(1-abcd p^{2n-1})(1-abcd p^{2n})},
 \\
 C_{n}&=\frac{a(1-p^{n})(1-bcp^{n-1})(1-bd p^{n-1})(1+cdp^{n-1})}{(1-abcd p^{2n-2})(1-abcd p^{2n-1})},
 \end{aligned}
\end{align*}
with $p=-q$. Upon comparing the recurrence relation \eqref{Recurrence-2} with that of the Askey-Wilson polynomials \cite{2010_Koekoek_&Lesky&Swarttouw}
\begin{align*}
 p_{n}(y;\mathfrak{a},\mathfrak{b},\mathfrak{c},\mathfrak{d}\rvert q)={}_4\varphi_{3}\left(
 \genfrac{}{}{0pt}{}{q^{-n}, \mathfrak{abcd}q^{n-1},\mathfrak{a} e^{i\theta}, \mathfrak{a}e^{-i\theta}}{\mathfrak{ab}, \mathfrak{ac}, \mathfrak{ad}}
 ; q,q\right),\quad y=\cos \theta,
\end{align*}
it is seen that the polynomials $Q_{n}(y;a,b,c,d\rvert q)$ can obtained from the Askey-Wilson polynomials by the formal substitutions
\begin{align*}
 e^{i\theta}\rightarrow i z,\quad \mathfrak{a}\rightarrow ia,\quad \mathfrak{b}\rightarrow ib,\quad \mathfrak{c}\rightarrow -ic,\quad \mathfrak{d}\rightarrow -i d,\quad q\rightarrow -q,
\end{align*}
where $i$ is the imaginary number. The polynomials $Q_{n}(x;a,b,c,d\rvert q)$ of degree $n$ in $x$ thus have the hypergeometric expression
\begin{align}
\label{Explicit}
 Q_{n}(x;a,b,c,d\rvert q)={}_4\varphi_{3}\left(
 \genfrac{}{}{0pt}{}{p^{-n}, abcdp^{n-1},-a z, az^{-1}}{-ab, ac, ad}
 ; p,p\right),
\end{align}
where $x=z-z^{-1}$ and where $p=-q$. 

The polynomials $Q_{n}(x;a,b,c,d\rvert q)$ satisfy a difference equation. Introduce the involution $\mathcal{I}_z$ defined by the action
\begin{align*}
 \mathcal{I}_z f(z)=f(z^{-1}),
\end{align*}
and let $\mathcal{D}_z$ be the divided-difference operator
\begin{align}
\label{Divided-Diff}
 \mathcal{D}_z=B(z)\,(T_{q}\mathcal{I}_z-1)+B(-z^{-1})\,(T_{q}^{-1}\mathcal{I}_z-1),
\end{align}
where $B(z)$ reads
\begin{align*}
 B(z)=\frac{(1+a z)(1+b z)(1-cz)(1-dz)}{(1+z^2)(1-q z^2)}.
\end{align*}
The operator \eqref{Divided-Diff} is very close to the Askey-Wilson operator, the main difference being the presence of the involution $\mathcal{I}_z$. A direct calculation using \eqref{Explicit} shows that the polynomials $Q_{n}(x;a,b,c,d\rvert q)$ satisfy the eigenvalue equation
\begin{align*}
 \mathcal{D}_{z}Q_{n}(x;a,b,c,d\rvert q)=\left[p^{-n}(1-p^{n})(1-abcd p^{n-1})\right]Q_{n}(x;a,b,c,d\rvert q).
\end{align*}
The operator $\mathcal{D}_z$ can be embedded in a realization of the $q$-deformed Bannai--Ito algebra \eqref{Rels}. If of takes
\begin{align}
\label{Realization}
\begin{aligned}
 \mathcal{J}_1&=\left(\frac{q^{1/2}}{(q-q^{-1})\sqrt{a b c d}}\right)\mathcal{D}_z+\left(\frac{q^{1/2}(q-a b c d)}{\sqrt{a b c d}(q^2-1)}\right),
 \\
 \mathcal{J}_2&=\frac{z-z^{-1}}{q-q^{-1}},
 \end{aligned}
\end{align}
it can be verified that one has
\begin{align*}
\mathcal{J}_2^2\mathcal{J}_1+(q+q^{-1})\mathcal{J}_2\mathcal{J}_1\mathcal{J}_2+\mathcal{J}_1\mathcal{J}_2^2&=\mathcal{J}_1+(q^{1/2}+q^{-1/2})\,\omega_3\,\mathcal{J}_2+\omega_1,
\\
\mathcal{J}_1^2\mathcal{J}_2+(q+q^{-1})\mathcal{J}_1\mathcal{J}_2\mathcal{J}_1+\mathcal{J}_2\mathcal{J}_1^2&=\mathcal{J}_2+(q^{1/2}+q^{-1/2})\,\omega_3\,\mathcal{J}_1+\omega_2,
\end{align*}
where the structure constants read
\begin{align*}
 \omega_1&=\frac{-q^{-1/2}(abcdq+abq^2-acq^2-bcq^2-adq^2-bdq^2+cdq^2+q^3)}{(1+q)(q-1)^2\sqrt{abcd}},
 \\
 \omega_2&=\frac{(a^2 bcdq+a b^2 cdq-abc^2d q-abcd^2 q-abcq^2-abd q^2+acd q^2+bcd q^2)}{(1+q)(q-1)^2 a b cd},
 \\
 \omega_3 &=\frac{-abc q-abdq +acd q+bcdq+aq^2+bq^2-cq^2-dq^2}{(1+q)(q-1)^2\sqrt{abcd}}.
\end{align*}
In the realization \eqref{Realization}, the Casimir operator \eqref{Casimir-q-BI} takes a definite value which is a complicated expression in the parameters $a$, $b$, $c$ and $d$.
\section{The $q\rightarrow 1$ limit}
\subsection{The $q\rightarrow 1$ limit of the Racah problem}
Consider the defining relations \eqref{First} of the $\mathfrak{osp}_{q}(1|2)$ algebra. In the $q\rightarrow 1$ limit, they take the form
\begin{align}
\label{lim1}
  [\widetilde{A}_0,\widetilde{P}]=0,\quad \{\widetilde{A}_{\pm},\widetilde{P}\}=0,\quad [\widetilde{A}_0,\widetilde{A}_{\pm}]=\pm \widetilde{A}_{\pm},\quad \{\widetilde{A}_{+},\widetilde{A}_{-}\}=2\widetilde{A}_0.
\end{align}
The relations \eqref{lim1} define the Lie superalgebra algebra $\mathfrak{osp}(1|2)$ extended by its grade involution, which is also referred to $sl_{-1}(2)$ \cite{2011_Tsujimoto&Vinet&Zhedanov_SIGMA_7_93}. In the same limit, the Casimir operator \eqref{OSPq-Casimir} reads
\begin{align}
\label{lim2}
 \widetilde{Q}=[\widetilde{A}_{+}\widetilde{A}_{-}-(\widetilde{A}_0-1/2)]\widetilde{P},
\end{align}
where the expression between the square brackets corresponds to the sCasimir of $\mathfrak{osp}(1|2)$ \cite{1995_Lesniewski_JMathPhys_36_1457}. The $q\rightarrow 1$ limit of the Hopf structure gives the coproduct
\begin{gather}
\begin{gathered}
\Delta(\widetilde{A}_0)=\widetilde{A}_0\otimes 1+1\otimes \widetilde{A}_0,
\\
\Delta(\widetilde{A}_{\pm})=\widetilde{A}_{\pm}\otimes \widetilde{P}+1\otimes \widetilde{A}_{\pm},\quad \Delta(\widetilde{P})=\widetilde{P}\otimes \widetilde{P},
\end{gathered}
\end{gather}
as well as the counit and coinverse 
\begin{gather}
\begin{gathered}
 \epsilon(\widetilde{P})=1,\quad \epsilon(\widetilde{A}_0)=0,\quad \epsilon(\widetilde{A}_{\pm})=0,
 \\
 \sigma(\widetilde{P})=\widetilde{P},\quad \sigma(\widetilde{A}_0)=-\widetilde{A}_0,\quad \sigma(\widetilde{A}_{\pm})=\widetilde{P}\widetilde{A}_{\pm},
 \end{gathered}
\end{gather}
as found in \cite{2000_Daskaloyannis&Kanakoglou&Tsohantjis_JMathPhys_41_652}.  The unitary $\mathfrak{osp}_q(1|2)$-modules $W^{(\epsilon,\mu)}$ also have a well-defined $q\rightarrow 1$ limit to unitary $\mathfrak{osp}(1|2)$-modules $V^{(\epsilon,\mu)}$. The actions \eqref{Actions} become
\begin{alignat}{2}
\label{dompe}
\begin{aligned}
 \widetilde{A}_0\,\ket{\epsilon,\mu;n}&=(n+\mu+1/2)\,\ket{\epsilon,\mu;n},\quad & \widetilde{P}\,\ket{\epsilon,\mu;n}&=\epsilon\, (-1)^{n}\,\ket{\epsilon,\mu;n},
 \\
 \widetilde{A}_{+}\ket{\epsilon,\mu;n}&=\sqrt{\widetilde{\sigma}_{n+1}}\,\ket{\epsilon,\mu;n+1},\quad & \widetilde{A}_{-}\,\ket{\epsilon,\mu;n}&=\sqrt{\widetilde{\sigma}_n}\,\ket{\epsilon,\mu;n-1},
 \end{aligned}
\end{alignat}
where $\widetilde{\sigma}_n=n+\mu(1-(-1)^{n})$. The modules $V^{(\epsilon,\mu)}$ associated to the actions \eqref{dompe} were the $\mathfrak{osp}(1|2)$-modules considered for the Racah problem in  \cite{2014_Genest&Vinet&Zhedanov_ProcAmMathSoc_142_1545}. The representations $V^{(\epsilon,\mu)}$ also have Bargmann realization on functions of argument $z$ defined by
\begin{gather}
\begin{gathered}
 \widetilde{A}_0(z)=z\partial_{z}+\mu+1/2,\quad \widetilde{P}(z)=\epsilon R_z,\quad \widetilde{A}_{+}(z)=z,
 \\
 \widetilde{A}_{-}(z)=\partial_{z}+\frac{\mu}{z}(1-R_{z}).
 \end{gathered}
\end{gather}
It is seen that in this realization $\widetilde{A}_{-}(z)$ coincides with the one-dimensional Dunkl derivative \cite{1989_Dunkl_TransAmerMathSoc_311_167}. The initial \eqref{Initial-Casimir}, intermediate \eqref{Intermediate-Casimir} and total \eqref{Total-Casimir} all have well defined limits when $q\rightarrow 1$. In this limit, the operators $\widetilde{I}_1=-\widetilde{Q}^{(23)}$ and $\widetilde{I}_3=-\widetilde{Q}^{(12)}$ satisfy the Bannai--Ito algebra relations
\begin{align}
\label{dompe2}
 \{\widetilde{I}_i,\widetilde{I}_j\}=\widetilde{I}_k+\omega_k,\quad  \omega_k=2(\mu_i\mu_j+\mu_k\mu),
\end{align}
where $(ijk)$ is an even permutation of $\{1,2,3\}$. The Casimir \eqref{Casimir-q-BI} reduces to
\begin{align}
 \widetilde{C}=\widetilde{I}_1^2+\widetilde{I}_2^2+\widetilde{I}_3^2,
\end{align}
and takes the value
\begin{align}
 \widetilde{C}=\mu_1^2+\mu_2^2+\mu_3^2+\mu^2-1/4.
\end{align}
These results are in accordance with those obtained in \cite{2014_Genest&Vinet&Zhedanov_ProcAmMathSoc_142_1545} and \cite{2015_Genest&Vinet&Zhedanov_CommMathPhys}. Adopting the same approach as the one used in this paper, one can obtain the spectra of the intermediate Casimir operators and construct the corresponding representations of \eqref{dompe2} to find the three-term recurrence relation satisfied by the Racah coefficients of $\mathfrak{osp}(1|2)$ and identify it with that of the Bannai--Ito polynomials.
\subsection{$q\rightarrow 1$ limit of the $q$-analogs of the Bannai--Ito polynomials}
Consider the polynomials defined by the recurrence relation \eqref{Recurrence-2}. Upon taking
\begin{align*}
 a=q^{2\rho_1+1/2},\quad b=-q^{-2r_2+1/2},\quad c=-q^{2\rho_2+1/2},\quad d=q^{-2r_1+1/2},\quad z=q^{x},
\end{align*}
dividing \eqref{Recurrence-2} by $(q-q^{-1})$ and taking the $q\rightarrow 1$ limit, one finds that the
the recurrence relation \eqref{Recurrence-2} becomes, in its normalized form,
\begin{align}
\label{dompe3}
 x\,\widetilde{Q}_{n}(x)=\widetilde{Q}_{n+1}(x)+(2\rho_1+1/2-\widetilde{A}_{n}-\widetilde{C}_n)\widetilde{Q}_{n}(x)+\widetilde{A}_{n-1}\widetilde{C}_{n}\widetilde{Q}_{n-1}(x),
\end{align}
where the coefficients read
\begin{align*}
 \widetilde{A}_{n}&=
 \begin{cases}
  \frac{(n+2\rho_1-2r_1+1)(n+2\rho_1-2r_2+1)}{2(n+\rho_1+\rho_2-r_1-r_2+1)} & \text{$n$ even}
  \\
  \frac{(n+2\rho_1+2\rho_2+1)(n+2\rho_1+2\rho_2-2r_1-2r_2+1)}{2(n+\rho_1+\rho_2-r_1-r_2+1)} & \text{$n$ odd}
 \end{cases},
 \\
  \widetilde{C}_{n}&=
 \begin{cases}
  -\frac{n(n-2r_1-2r_2)}{2(n+\rho_1+\rho_2-r_1-r_2)} & \text{$n$ even}
  \\
  -\frac{(n+2\rho_2-2r_1)(n+2\rho_2-2r_2)}{2(n+\rho_1+\rho_2-r_1-r_2)} & \text{$n$ odd}
 \end{cases}.
\end{align*}
Comparing with the recurrence relation \eqref{BI-Recurrence} satisfied by the Bannai--Ito polynomials, one sees from \eqref{dompe3} that $ \widetilde{Q}_n(x)=2^{n}B_{n}(\frac{x-1/2}{2})$. In consequence, the polynomials $Q_{n}(x;a,b,c,d\rvert q)$ defined by the recurrence relation \eqref{Recurrence-2} are $q$-analogs of the Bannai--Ito polynomials. Similarly, upon taking the limit when $q\rightarrow 1$ of the divided-difference operator \eqref{Divided-Diff} with parametrization \eqref{dompe3}, one finds
\begin{multline*}
 \lim_{q\rightarrow 1}\frac{\mathcal{D}_z}{q-q^{-1}}=\left(\frac{(x-2\rho_1-1/2)(x-2\rho_2-1/2)}{2x-1}\right)(T^{-}R-1)
 \\
 -\left(\frac{(x-2r_1+1/2)(x-2r_2+1/2)}{2x+1}\right)\left(T^{+}R-1\right),
\end{multline*}
which corresponds to \eqref{BI-Operator}, up to an affine transformation and a change of variable.
\section{Conclusion}
In this paper, the Racah problem for the quantum superalgebra $\mathfrak{osp}_q(1|2)$ was considered and a family of basic orthogonal polynomials that generalize the Bannai--Ito polynomials was proposed.  While these $q$-analogs of the Bannai--Ito polynomials are formally related to the Askey-Wilson polynomials, the two families of (truncated) polynomials exhibit different algebraic properties, the former arising in the Racah coefficients for the quantum superalgebra $\mathfrak{osp}_q(1|2)$ and the latter arising in the Racah coefficients for the quantum algebra $\mathfrak{sl}_q(2)$.

The results presented here and those of \cite{2011_Tsujimoto&Vinet&Zhedanov_SIGMA_7_93} suggest a connection between quantum superalgebras and quantum algebras when $q\rightarrow -q$; see also \cite{1992_Zachos_ModPhysLettA_7_1595}. Also of interest is the investigation of the transformation $q\rightarrow -q$ and its consequences for other families of polynomials of the Askey scheme. We plan to report on this in the near future.

\section*{Acknowledgements}
\noindent
VXG holds an Alexander-Graham-Bell fellowship from the Natural Science and Engineering Research Council of Canada (NSERC). The research of LV is supported in part by NSERC. AZ acknowledges the hospitality of the Centre de recherches math\'ematiques (CRM).


\end{document}